\documentclass[reqno,11pt]{amsart}
\usepackage{amsmath,amssymb,amsthm,mathrsfs}
\usepackage{epsfig,color}
\usepackage[latin1]{inputenc}
\usepackage[english]{babel}
\usepackage[numbers]{natbib}

\usepackage{esint}

\usepackage{hyperref}

 \allowdisplaybreaks
\numberwithin{equation}{section}

\def\R{\mathbb R}

\newcommand{\dist}{\mathop{\mathrm{dist}}}

\def\trace{{\rm tr}}
\def\dist{{\rm dist}}

\def\00{{\bf 0}}

\newcommand{\tr}{\mbox{tr}\,}

\newcommand{\diver}{{\rm div \,}}

\newtheorem*{theorem*}{Theorem}
\newtheorem{theorem}{Theorem}[section]
\newtheorem{lemma}[theorem]{Lemma}
\newtheorem{proposition}[theorem]{Proposition}

\newtheorem{corollary}[theorem]{Corollary}

\newtheorem{definition}[theorem]{Definition}

\begin{document}
  
    \title[]{Classification and non-existence results for weak solutions to quasilinear elliptic equations with Neumann or Robin boundary conditions}
    
  \date{}

\author{Giulio Ciraolo}
\address{G. Ciraolo. Dipartimento di Matematica "Federigo Enriques",
Universit\`a degli Studi di Milano, Via Cesare Saldini 50, 20133 Milano, Italy}
\email{giulio.ciraolo@unimi.it}

\author{Rosario Corso}
\address{R. Corso. Dipartimento di Matematica e Informatica,
	Universit\`a degli Studi di Palermo, Via Archirafi 34, 90123 Palermo, Italy
}
\email{rosario.corso02@unipa.it}

\author{Alberto Roncoroni} 
\address{A. Roncoroni. Dipartimento di Matematica e Informatica ``Ulisse Dini'', Universit\`a degli Studi di Firenze, Viale Morgagni 67/A, 50134 Firenze, Italy}
\email{alberto.roncoroni@unifi.it}


\begin{abstract}
We classify positive solutions to a class of quasilinear equations with Neumann or Robin boundary conditions in convex domains. Our main tool is an integral formula involving the trace of some relevant quantities for the problem. 

Under a suitable condition on the nonlinearity, a relevant consequence of our results is that we can extend to weak solutions a celebrated result obtained for stable solutions by Casten and Holland and by Matano.
\end{abstract}

\maketitle

%\tableofcontents

\section{Introduction}
Given a bounded domain $\Omega \subset \mathbb{R}^n$, $n\geq 2$, we consider solutions to quasilinear equations of the form
\begin{equation} \label{p_lapl_eq}
\Delta_p u + f(u) = 0 \quad \text{ in } \Omega \,,
\end{equation}
with $1<p<n$ and where 
$$\Delta_p u := \diver(|\nabla u|^{p-2} \nabla u ) $$ is the so-called $p$-Laplace operator. Our main goal is to give classification and non-existence results for \eqref{p_lapl_eq} (and for more general quasilinear operators) once Neumann or Robin boundary conditions are prescribed at the boundary of $\Omega$. 

The study of non-existence results for \eqref{p_lapl_eq} under Neumann boundary conditions has a long history in the PDE's community, which started with the celebrated papers of Casten and Holland \cite{CastenHolland} and Matano \cite{Matano}: if $\Omega$ is a convex domain, then any non-constant solution to
\begin{equation} \label{pb_neum_semilinear}
\begin{cases}
\Delta u + f(u) = 0 & \text{ in } \Omega \\
\partial_\nu u = 0 & \text{ on } \partial \Omega
\end{cases}
\end{equation}
is, if it exists, unstable.
%\footnote{We recall that a solution to 
%$$
%\Delta u+f(u)=0 \quad \text{in $\Omega$}\, ,
%$$
%is called \emph{stable} in $\Omega$ if the second variation of the energy functional
%$$
%\mathcal{E}(u)=\frac{1}{2}\int_{\Omega}|\nabla u|^2\, dx-\int_\Omega F(u)\, dx\, , \quad \text{where $F(t):=\int_0^tf(s)\, ds$}\, ,
%$$
%is nonnegative; explicitly
%$$
%\int_\Omega f'(u)\varphi\, dx\leq \int_\Omega |\nabla \varphi|^2\, dx \quad \text{ for all $\varphi\in C^{\infty}_c(\Omega)$}\, .
%$$}. 
This result was extended to more general settings as Riemannian manifolds and to classical solutions to more general operators in \cite{BMMP,BPT,DPV,DPV2,Jimbo} (see also \cite{Cabre_Poggesi} and \cite{Dupaigne} for a general reference on stable solutions).

Our main goal is to prove classification and non-existence results for Neumann or Robin type boundary value problems involving the equation
\begin{equation} \label{eq_general}
\diver(a(\nabla u))  + f(u) = 0,
\end{equation}
with 
\begin{equation} \label{a_def}
a(\xi)=H(\xi)^{p-1} \nabla H(\xi)\qquad \forall\,\xi \in \R^n,
\end{equation}
where $H$ is a norm in $\mathbb{R}^n$ (see Subsection \ref{subsect_notation} for more details), $1<p<n$ and $\Omega\subset \mathbb{R}^n$ is a bounded convex domain. Notice that if the norm $H(\cdot)$ is the Euclidean norm $|\cdot |$ then $\diver(a(\nabla u))$ is the usual $p$-Laplace operator and equation \eqref{eq_general} reduces to \eqref{p_lapl_eq}. More generally, in this paper we assume that $H$ is a norm of $\mathbb R^n$ such that $H^2$ is of class $C^{2}(\R^n\setminus \{\mathcal O\})$ and it is uniformly convex and $C^{1,1}$ in $\R^n$, i.e. there exist constants $0<\lambda \leq \Lambda$ such that
\begin{equation}
\label{eq:elliptic H}
\lambda {\rm Id}\leq H(\xi)\,D^2H(\xi)+\nabla H(\xi)\otimes \nabla H(\xi) \leq \Lambda\,{\rm Id}\qquad \forall\,\xi \in \R^n\setminus \{\mathcal O\} \,.
\end{equation}

We consider boundary value problems having Neumann or Robin type conditions at the boundary. More precisely, we shall assume that 
\begin{equation}\label{bc_general}
a(\nabla u) \cdot \nu + h(u) = 0 \quad \text{ on } \partial \Omega \,,
\end{equation}
where $h$ satisfies some assumption to be specified later. 

In this paper, we are not considering stable solutions; instead we are considering a general weak solution to \eqref{eq_general} with boundary condition \eqref{bc_general}. 

\begin{definition} We say that $u\in W^{1,p}_{loc}(\Omega)$ is a weak solution to \eqref{eq_general} and \eqref{bc_general} if $f(u) \in L^1_{loc}(\Omega)$ and it satisfies
\begin{equation}\label{def_debole}
- \int_{\partial \Omega} h(u) \phi d\sigma - \int_\Omega a(\nabla u) \cdot \nabla \phi \, dx + \int_\Omega f(u) \phi \,dx= 0 \quad \forall \phi \in C_c^1(\mathbb{R}^n) \,.
\end{equation}
\end{definition}

\medskip

Our first main result is a classification result for a Neumann problem for positive weak solutions to \eqref{eq_general} in a convex domain.

\begin{theorem} \label{thm_main}
Let $\Omega\subset \mathbb{R}^n$ be a bounded convex domain and let $H$ be a norm in $\mathbb{R}^n$ such that $H^2 \in C^{2}(\R^n\setminus \{\mathcal O\}) \cap C^{1,1}(\R^n)$ and satisfies \eqref{eq:elliptic H}. Let $f \in C^1(\mathbb R)$ be such that the function 
\begin{equation} \label{Phi_def} 
\Phi(t) := \frac{f(t)}{t^{p^*-1}} \quad \textmd{ is nonincreasing. }
\end{equation}
Then there exist no positive bounded weak solutions $u$ to
\begin{equation} \label{pb_main}
\begin{cases}
\diver(a(\nabla u))  + f(u) = 0 & \text{ in }\Omega \\
a(\nabla u ) \cdot \nu = 0 & \text{ on } \partial \Omega \,
\end{cases}
\end{equation}
unless $u$ is constant.
\end{theorem}

Theorem \ref{thm_main} extends the results in \cite{CastenHolland} and \cite{Matano} to a general weak solution (not necessarily stable). The condition needed for such generalization is \eqref{Phi_def}. This condition was already used in \cite{Bianchi} to prove, by using the method of moving planes, an analogous non-existence result in $\mathbb{R}^n$ and not in a bounded convex domain. The method of moving planes is not suitable for the problem that we consider. Indeed, the anisotropic setting and the fact that we are dealing with a problem on a bounded domain with Neumann boundary condition is an obstruction to a standard application of the method of moving planes (see for instance \cite{Lin}). 

We also notice that condition \eqref{Phi_def} seems to be optimal, as follows from many results related to the Lin-Ni conjecture \cite{Lin-Ni}. In this conjecture one considers positive solutions to 
\begin{equation} \label{Lin_Ni}
\begin{cases}
\Delta u - \lambda u + u^q = 0 & \textmd{ in }\Omega \\
\partial_\nu u = 0 & \textmd{ on }\partial \Omega
\end{cases}
\end{equation} 
(notice that the constant function $u= \lambda^{\frac{1}{q-1}}$ is a solution to this problem). The conjecture of Lin-Ni is the following: there exists a constant $\lambda_0>0$ such that if $0<\lambda < \lambda_0$ then \eqref{Lin_Ni} admits only the constant solution. This conjecture is true if $1<q<2^*-1$ (see \cite{Lin-Ni-Taka} and \cite{Li-Zhao}). In the critical case $q=2^*-1$, the conjecture is in general false \cite{Wang_TAMS}. If $\Omega$ is a convex domain, the conjecture is still false for $N \geq 4$ as proved in \cite{WangII}, but it is true for $N=3$ (see \cite{Wei,Zhu}). Since Theorem \ref{thm_main1} proves the non-existence of nonconstant solutions for $\lambda =0$, the mentioned examples show that one can not improve condition \eqref{Phi_def} in Theorem \ref{thm_main1}, since counterexamples are available for linear small perturbations of $f(u)=u^{2^*-1}$ (for which \eqref{Phi_def} fails).

The technique that we use does not rely upon maximum principle and it is more in the spirit of the results of \cite{GS}, \cite{SerZou} and \cite{Veron}, where non-existence results for semilinear and quasilinear equations where obtained in $\mathbb{R}^n$ and compact Riemannian manifolds. More precisely, Theorem \ref{thm_main} is a consequence of a general integral inequality which holds for bounded positive weak solutions to \eqref{def_debole}. In particular we have the following proposition.

\begin{proposition} \label{prop_int_ineq}
Let $\Omega$ be a bounded convex domain with boundary of class $C^2$ and let $H$ be as in Theorem \ref{thm_main}. Let $u \in W^{1,p}_{loc}(\Omega)$ be a bounded weak solution to 
\begin{equation} \label{pb_main2}
\begin{cases}
\diver(a(\nabla u))  + f(u) = 0 & \text{ in }\Omega \\
a(\nabla u ) \cdot \nu+h(u) = 0 & \text{ on } \partial \Omega \,
\end{cases}
\end{equation}
where $f,h\in C^1(\R)$. 
 Then 
\begin{multline} \label{integral_ineq}
(n-1) \int_\Omega u^{\frac{n}{n-p}} H^p(\nabla u) \Phi'(u)\, dx \geq   \int_{\partial \Omega} u^{\frac{n}{n-p} - p^* +1} f(u) h(u)\, d\sigma \\ 
+ n  \int_{\partial \Omega} u^{\frac{n}{n-p} - p^* +1} H^p(\nabla u) \left(  h'(u) - \frac{(p-1)(n-1)}{n-p}\frac{h(u)}{u} \right)\, d\sigma \\ 
+ n \int_{\partial \Omega} u^{\frac{n}{n-p} - p^* +1} \mathrm{II}(a^T(\nabla u),a^T(\nabla u))\, d\sigma \, ,
\end{multline}
where $p^\ast=np/(n-p)$ is the Sobolev conjugate of $p$, $\mathrm{II}(\cdot, \cdot)$ denotes the second fundamental form of $\partial\Omega$ and $a^T(\nabla u)$ is the tangential component of $a(\nabla u)$.

Moreover, if the equality in \eqref{integral_ineq} is attained then either $u$ is constant or there exists $a,b > 0$ and $x_0 \in \overline{\Omega}$ such that
\begin{equation} \label{Talenti_p}
u(x) = \left(a + b H_0(x_0-x)^{\frac{p}{p-1}} \right)^{-\frac{n-p}{p}}
\end{equation}
for any $x \in \overline \Omega$.
\end{proposition}

Proposition \ref{prop_int_ineq} is our main tool to prove Theorem \ref{thm_main}. It is clear that, up to an approximation argument for $\Omega$, Theorem \ref{thm_main} directly follows from Proposition \ref{prop_int_ineq} since $h(u) \equiv0$ and the second fundamental form of $\partial \Omega$ is nonnegative definite. 

We notice that Theorem \ref{thm_main} is trivial if $f\geq0$, as it can be easily verified by integrating the equation in \eqref{pb_neum_semilinear} and using the Neumann boundary condition. Hence, the result is of interest when $f$ is negative somewhere. In the case of more general Robin boundary conditions, we can still prove a classification result by exploiting \eqref{integral_ineq} again. In particular, we can consider Robin type boundary conditions under the assumption that  $f(u) \geq 0$ on $\partial \Omega$ and $h$ satisfies
\begin{equation} \label{h_condition}
h(t)\geq 0 \quad \text{and}\quad  h'(t) - \frac{(p-1)(n-1)}{n-p}\frac{h(t)}{t} \geq 0 \quad \forall t >0\,.
\end{equation}
%which clearly implies a polynomial growth on $h$ of the form
%\begin{equation} \label{h_hp}
%h(t) \geq C t^{\frac{(p-1)(n-1)}{n-p}} \quad \forall t>0\,,
%\end{equation}
%for some $C \geq 0$. We have the following result.

\begin{theorem} \label{thm_main1}
Let $\Omega\subset \mathbb{R}^n$ be a bounded convex domain and let $H$ be as in Theorem \ref{thm_main}. Let $f,h$ be $C^1$ functions satisfying \eqref{Phi_def} and  \eqref{h_condition}, respectively. Then there exist no positive bounded weak solutions $u$ to
\begin{equation} \label{pb_main_robin}
\begin{cases}
\diver(a(\nabla u))  + f(u) = 0 & \text{ in }\Omega \,\\
a(\nabla u ) \cdot \nu + h(u) = 0 & \text{ on } \partial \Omega \,
\end{cases}
\end{equation}
such that $f(u) \geq 0$ on $\partial \Omega $ unless $u$ is constant.
\end{theorem}

Now we give some comment on the technique used for proving Proposition \ref{prop_int_ineq}. As we already mentioned, our approach is based on integral identities and it is inspired by \cite{GS} and \cite{SerZou}. The main idea consists of considering a suitable vector field involving $u$ and its derivatives and to prove some integral identity and find the final integral inequality \eqref{integral_ineq} by using just one pointwise inequality. 

Since we are concerned with $p$-Laplace type operators, there are several technical difficulties that we have to tackle and, due to the lack of regularity of the solution, we have to argue by approximation. 
In this direction, a crucial assumption is that $u$ is bounded. The boundedness of solutions of semilinear and quasilinear elliptic equations is a very challenging problem. Recently, in \cite{FCRS} it has been shown that stable solutions to \eqref{eq_general} are bounded for $p=2$ up to dimension 9 (see also \cite{Miraglio} for a general $p$). If one does not look for stable solutions, then it has been proved in \cite{Serrin_local} and \cite{Peral} that weak solutions are locally bounded in the critical and subcritical case. Following the approach in  \cite{Serrin_local} and \cite{Peral} we provide a global $L^\infty$ bound on the solutions under the assumption that $f$ is critical or subcritical and, as a consequence, we obtain the following classification result.

\begin{corollary} \label{corollary_critical}
Let $\Omega$, $f$ and $H$ satisfy the assumptions of Theorem \ref{thm_main}. If there exists a constant $C\geq 0$ such that 
\begin{equation} \label{f_critical_growth}
|f(t)| \leq C(1 + t )^{p^*-1}
\end{equation}
for some positive constant $C$, then there exists no positive nonconstant weak solutions to \eqref{pb_main}.

Analogously, if $\Omega$, $f$, $h$ and $H$ satisfy the assumptions of Theorem \ref{thm_main} and $f$ satisfies \eqref{f_critical_growth}, then there exists no positive nonconstant weak solutions to \eqref{pb_main_robin}.
\end{corollary}

We conclude this introduction with a remark on a non-existence result which is related to overdetermined problems where one assumes $|\nabla u|=0$ on $\partial \Omega$.\footnote{This kind of overdetermined condition comes from the well-known Schiffer's conjecture, which asserts that the ball is the only bounded domain such that a Neumann eigenfunction of the Laplacian is constant at the boundary.} 

\begin{theorem} \label{thm_schiffer}
Let $\Omega\subset \mathbb{R}^n$ be a bounded convex domain, let $H$ be as in Theorem \ref{thm_main1}, and let $f \in C^1(\mathbb R)$ be such that \eqref{Phi_def} is satisfied. Then there exists no positive bounded weak solutions to
\begin{equation}\label{pb_schiffer}
\begin{cases}
\diver(a(\nabla u))  + f(u) = 0 & \text{ in }\Omega \,\\
|\nabla u| = 0 & \text{ on } \partial \Omega \,,
\end{cases}
\end{equation}
unless $u$ is constant.
\end{theorem}
We notice that the condition $|\nabla u|=0$ on $\partial \Omega$ implies that $u$ is constant and $a(\nabla u) \cdot \nu =0$ on $\partial \Omega$ and hence Theorem \ref{thm_schiffer} easily follows from Proposition \ref{prop_int_ineq}. 

\medskip

The paper is organized as follows. In Section \ref{section_regularity} we introduce some notation and provide a global $L^\infty$  bound on the solution under the assumption that $f$ is critical or subcritical (this result is needed only to prove Corollary \ref{corollary_critical}) and we prove a higher order integrability result for bounded weak solutions $u$ to \eqref{pb_main_robin}. In Section \ref{section_prop} we prove Proposition \ref{prop_int_ineq}. In Section \ref{section_thm} we prove the main theorems.

\subsection*{Acknowledgments} The authors thank Louis Dupaigne and Alberto Farina for useful discussions and remarks.

The authors have been partially supported by the ``Gruppo Nazionale per l'Analisi Matematica, la Probabilit\`a e le loro Applicazioni'' (GNAMPA) of the ``Istituto Nazionale di Alta Matematica'' (INdAM, Italy). R.C. has been partially supported by the PRIN 2017 project ``Qualitative and quantitative aspects of nonlinear PDEs''. A.R. has been partially supported by the PRIN 2015 project ``Partial differential equations and related analytic geometric inequalities'' and by the PRIN 2017 project ``Direct and inverse problems for partial differential equations: theoretical aspects and applications''.

\section{Notation and preliminary results} \label{section_regularity}
In this section we introduce some notation and give two preliminary results.

\subsection{Notation} \label{subsect_notation}
Let $\Omega \subset \mathbb{R}^n$ be a bounded domain. We denote by $B_r(x)$ the usual Euclidean ball centered at $x$ and of radius $r$.

We consider $\mathbb{R}^n$ endowed with a ``norm'' or gauge $H:\mathbb{R}^n\rightarrow\mathbb{R}$ such that 
\begin{itemize}
\item $H$ is convex;
\item $H$ is positively one-homogeneous, i.e. $H(\lambda\xi)=\lambda H(\xi)$ for all $\lambda>0$ and $\xi\in\mathbb{R}^n$;
\item $H(\xi)>0$ for all $\xi\in\mathbb{S}^{n-1}$.
\end{itemize}
Observe that we do not require $H$ to be symmetric, so it may happen that $H(\xi) \neq H(-\xi)$. In this paper we assume that $H^2$ is of class $C^{2}(\R^n\setminus \{\mathcal O\})$ and it is uniformly convex and $C^{1,1}$ in $\R^n$ (hence \eqref{eq:elliptic H} holds).

An important property of $H$, which follows from the homogeneity, is the following 
$$
\nabla H(\xi)\cdot\xi=H(\xi)\quad \text{ for all $\xi\in\mathbb{R}^n$}\, .
$$
In particular we have the following relation, which we will use several times, 
$$
a(\nabla u)\cdot\nabla u=H^{p}(\nabla u)\, .
$$
Moreover, we denote by $H_0$ the dual norm associated to $H$:
$$
H_0(x)=\sup_{H(\xi)=1}\xi\cdot x\quad \text{ for all $x\in\mathbb{R}^n$}\, .
$$

\subsection{Boundedness of solutions in the critical and subcritical case} 
We start by proving $L^\infty$ bounds for solutions to \eqref{pb_main_robin}. Even if this result is needed only in Corollary \ref{corollary_critical}, we prefer to start from this result in order to introduce some notation and approximation argument which will be needed in the rest of the paper.  

Bounds on the $L^\infty$ norm of solutions \eqref{pb_main_robin} follows from the growth assumption \eqref{f_critical_growth} by following \cite{Serrin_local} and \cite{Peral} (see also \cite[Lemma 2.1]{CFR} and \cite{Harrabi}). Since we aim at giving global bounds and have to deal with a Robin type boundary condition, we give a proof for the sake of completeness.

\begin{lemma}\label{Lemma_Cianchi_prel}
	Let $\Omega\subseteq \mathbb{R}^n$ be a  bounded domain, $f,h:(0,\infty)\to \R$ such that 
	\begin{equation}\label{f_bound}
	|f(t)|\leq C(1+t)^{p^*-1}
	\end{equation}
	for some $C>0$ and every $t>0$,  
	$h\geq 0$ and let $u\in W^{1,p}_{loc}(\Omega)$ % \mathcal{D}^{1,\alpha,p}(\Sigma):= \left\{ u \in L^{\alpha}(\Sigma)\cap W^{1,p}_{loc}(\Sigma) \, :\ \nabla u \in L^{p}(\Sigma) \right\}
	be a positive solution to 
	\begin{equation}\label{approx_ter}
	\begin{cases}
	\diver(a(\nabla u)) + f(u) = 0 & \text{ in } \Omega \\
	a(\nabla u)\cdot\nu+h(u)=0 & \text{ on } \partial\Omega\, ,
	\end{cases}
	\end{equation}
	where the $a:\mathbb{R}^n\rightarrow\mathbb{R}^n$ is a continuous vector field such that the following holds: 
	%satisfies:\\
	%	- $a\in C^1(\mathbb{R}^n)$ if $p\geq 2$;\\
	%	- $a\in C^1(\mathbb{R}^n\setminus\{0\})$ if $1<p<2$.\\
	%Also, assume that
	there exist $\beta>0$ and $0\leq s\leq 1/2$ such that
	\begin{equation}\label{ellipticity_approx}
	|a(\xi)|\leq \beta(|\xi|^2+s^2)^{\frac{p-1}{2}} \quad \text{and} \quad \xi\cdot a(\xi)\geq \dfrac{1}{\beta}\int_0^1\left( t^2|\xi|^2+s^2\right)^{\frac{p-2}{2}}|\xi|^2\, dt  \, , 
	\end{equation}
	for every $\xi\in\mathbb{R}^n$. Then $u\in L^{\infty}(\Omega)$ and 
	\begin{equation*}
	\|u\|_{\infty}\leq K (\|u\|_{p}+s). 
	\end{equation*}	 
	where $K$ depends only on $n$, $p$, $\beta$ and on the Sobolev constant of $\Omega$. 
	\end{lemma}

\begin{proof}
	%  \todo{[DA TOGLIERE DA QUI]}  We can rewrite the equation satisfied by $u$ as follows:
	%	$$
	%	-\diver(a(\nabla u))=f(x)u^{p-1}+g(x)
	%	$$ 
	%	where 
	%	$$
	%	f(x)=\begin{cases}
	%	0 & \text{ if } u< 1 \\
	%	u^{\alpha-p} & \text{ if } u\geq 1  \, ,
	%	\end{cases}
	%	$$
	%	and 
	%	$$
	%	g(x)=\begin{cases}
	%	0 & \text{ if } u> 1 \\
	%	u^{\alpha-1} & \text{ if } u\leq 1  \, .
	%	\end{cases}
	%	$$
	%	Clearly $g\in L^{\infty}$. 	We also have that $f\in L_{loc}^r$ for some $r>\frac{n}{p}$. %We distinguish between the cases $\alpha-p+1\leq0$ and $\alpha-p+1>0$. 
	%	This is certainly true if $\alpha-p+1\leq0$, because $f\in L^\infty$. On the other hand, if $\alpha-p>0$ then $f\in L_{loc}^r$ with $r=\frac{\alpha}{\alpha-p}>\frac{n}{p}$. \\
	%	Hence the proof follows as in %\cite[Theorem E.0.20]{Peral} and
	%	\cite[Theorem 1]{Serrin_local}. In particular, when we take a test-function $\eta$ and integrate by parts in \eqref{approx_ter}
	%	\begin{equation}
	%	- \int_{\partial \Omega} h(u) \phi d\sigma - \int_\Omega a(\nabla u) \cdot \nabla \phi \, dx + \int_\Omega f(u) \phi \,dx= 0 \, ,
	%	\end{equation}
	%	we still choose $\eta\in C^\infty_0(\mathbb{R}^n)$ even if we are in the conical setting. The reason is that the term on the boundary is null thanks to \eqref{approx_ter}.
	%	%	
	%	\todo{[TOGLIERE FINO A QUI]}
	%	
	%	
	%	\todo{***********************************}
	%	
	We give a sketch of this proof by following the one of \cite[Theorem 1]{Serrin_local} (see also \cite[Theorem E.0.20]{Peral} and  \cite[Lemma 2.1]{CFR}). We first notice that 
	\begin{equation}\label{eq:claim}
	\xi\cdot a(\xi)\geq \beta_* \left(|\xi|^p - s^p\right) \,.
	\end{equation}
	where 
	$$
	\beta_* = \frac{1}{\beta} \min \left( \dfrac{1}{(p-1)}, 1\right) \,.
	$$
	Indeed, if $p\geq 2$, then \eqref{ellipticity_approx} implies
	$$
	\xi\cdot a(\xi)\geq \dfrac{1}{\beta (p-1)}|\xi|^p\,.
	$$ 
	Equation \eqref{eq:claim} is also true if $1<p<2$ 
	%	 then \eqref{17.30} is obtained by using a more careful argument. 
	%	 We claim that
	%	 \begin{equation}\label{eq:claim}
	%	 \int_0^1\left( t^2|\xi|^2+s^2\right)^{\frac{p-2}{2}}|\xi|^2\, dt \geq \frac12\left(|\xi|^p - s^p\right).
	%	 \end{equation}
	%	 To prove this we consider two cases. If 
	and $s > |\xi|$, since the right-hand side of \eqref{eq:claim} is negative. It remains to prove \eqref{eq:claim} when $1<p<2$ and $s \leq |\xi|$. In this case
	$$
	t^2|\xi|^2+s^2 \leq 2|\xi|^2 \qquad \text{ for $t \in [0,1]$}, 
	$$
	and hence
	$$
	\int_0^1\left( t^2|\xi|^2+s^2\right)^{\frac{p-2}{2}}|\xi|^2\, dt \geq   |\xi|^p,
	$$
	that again implies \eqref{eq:claim}. 
	
	Let $\tilde{u}=u+s$, and we obtain that $\tilde u$ satisfies 
	\begin{equation} \label{starting}
	|a(\nabla \tilde u)|\leq \beta_*(|\nabla \tilde u|^2+\tilde{u}^2)^{\frac{p-1}{2}} \quad \text{and} \quad \nabla \tilde u\cdot a(\nabla \tilde u)\geq \dfrac{1}{2\beta_*}\left(|\nabla \tilde u|^p - \tilde{u}^p\right), 
	\end{equation}
	for every $\xi\in\mathbb{R}^n$, which are our starting point. In order to avoid heavy notation, we write $u$ instead of $\tilde u$. 
	
\medskip
	
\noindent \emph{Step 1: $u\in L^{q p}_{loc}(\Omega)$ (with $q \geq 1$) implies $u\in L^{q p^\ast}_{loc}(\Omega)$.} Given $l>0$ and $1\leq q$, we define
	\begin{equation}\label{F_Peral}
	F(u)=\begin{cases}
	u^{q}\ & \text{ if } u\leq l \\
	q l^{q-1}(u-l)+l^q & \text{ if } u>l  \, ,
	\end{cases}
	\end{equation}
	and 
	$$
	G(u)=\begin{cases}
	u^{(q-1)p+1}\ & \text{ if } u\leq l \\
	((q-1)p+1)l^{(q-1)p}(u-l)+ l^{(q-1)p+1} & \text{ if } u>l  \,.
	\end{cases}
	$$
	Let 
	$$
	\xi=\eta^p G(u)
	$$
	where $\eta\in C^\infty_c(\mathbb{R}^n)$ and $\eta\geq 0$. From \eqref{approx_ter} with $\xi$ used as test-function, we obtain
	\begin{equation}\label{debole}
	\int_{\partial \Omega} h(u)\eta^p G(u) d\sigma +\int_{\Omega} {a}(\nabla u)\cdot\nabla (\eta^p G(u))\, dx=\int_{\Omega}f(u)\eta^p G(u)\, dx\, .
	\end{equation}
	From \eqref{debole}, \eqref{eq:claim} and by the fact that $h,G\geq 0$ we get 
	\begin{equation*}
	\begin{aligned}
	c_1\int_{\Omega} \eta^p G'(u)|\nabla u|^p\, dx \leq  \int_{\Omega}\eta^{p-1}G(u)|a(\nabla u)\cdot\nabla\eta|\, dx+\int_{\Omega}f(u)\eta^p G(u)\, dx 
	+\int_\Omega u^p\eta^p G'(u)\, dx \, ,
	\end{aligned}
	\end{equation*}
	for some $c_1>0$. We estimate the second term by using Young's inequality and  \eqref{ellipticity_approx}, and we obtain
	\begin{equation*}
	\begin{aligned}
	\eta^{p-1} |a(\nabla u)\cdot\nabla\eta|&\leq
	\epsilon^{\frac{p}{p-1}}u^{-1}|a(\nabla u)|^{\frac{p}{p-1}}\eta^p+\epsilon^{-p}u^{p-1}|\nabla\eta|^p\\
	&\leq C_1\epsilon^{\frac{p}{p-1}}u^{-1}(|\nabla u|^p+u^p)\eta^p+\epsilon^{-p}u^{p-1}|\nabla\eta|^p,
	\end{aligned}
	\end{equation*}
	for any $\epsilon\in (0,1)$, where $C_1$ depends only on $\beta$ and $p$.
	From \eqref{f_bound}, since $G(u)\leq uG'(u)$ and $G$ is convex, we obtain 
	\begin{equation*}
	\begin{aligned}
	c_1\int_{\Omega} \eta^p G'(u)|\nabla u|^p\, dx &\leq C_1\epsilon^{\frac{p}{p-1}} \int_{\Omega} \eta^p G'(u)|\nabla u|^p\, dx
	+   C_2 \int_{\Omega} u^p \eta^p G'(u)\, dx \\
	&+\epsilon^{-p}\int_{\Omega}G(u)u^{p-1}|\nabla\eta|^p\, dx+ \int_{\Omega}f(u)\eta^p G(u)\, dx\, ,
	\end{aligned}
	\end{equation*}
	 for any $\epsilon\in (0,1)$ and for some constant $C_2$ which depends only on $\beta$ and $p$. We choose $\epsilon$ small enough and obtain
	$$
	c_2\int_{\Omega} \eta^p G'(u)|\nabla u|^p\, dx \leq \int_{\Omega}  \eta^p u^p G'(u)\, dx \\
	+\int_{\Omega}G(u)u^{p-1}|\nabla\eta|^p\, dx+ \int_{\Omega}f(u)\eta^p G(u)\, dx\, ,
	$$
	where $c_2>0$ depends only on $n$, $p$ and  $\beta$.
	Since 
	$$
	G'(u)\geq c[F']^p
	$$ 
	and
	$$
	u^{p-1}G(u)\leq C[F(u)]^p \,,
	$$ 
	we obtain  
	\begin{equation*}
	c_3\int_{\Omega}|\nabla (\eta F(u))|^p\, dx \\ \leq \int_{\Omega}\eta^pu^p G'(u)\, dx+\int_{\Omega}|\nabla\eta|^pF^p(u)\, dx+ \int_{\Omega}\eta^pf(u)u^{1-p} F^p(u)\, dx\, ,
	\end{equation*}
	where $c_3$ depends only on $n$, $p$ and  $\beta$.
	Hence, thanks to the classical Sobolev inequality in bounded domains \cite[Lemma 5.10]{Adams}
	$$
	C_* \|u\|_{p^*} \leq \|u\|_p + \|\nabla u\|_p \,,
	$$ 
	and thanks to \eqref{f_bound}, we find
	\begin{equation}\label{post_sobolev}
	c_4\left(\int_{\Omega}F^{p^\ast}(u)\eta^{p^\ast}\, dx\right)^{\frac{p}{p^\ast}}  \leq \int_{\Omega}\eta^pu^pG'(u)\, dx+\int_{\Omega}(\eta^p+|\nabla\eta|^p)F^p(u)\, dx + \int_{\Omega}\eta^pu^{p^{\ast}-p} F^p(u)\, dx \, ,
	\end{equation}
	where $c_4>0$ depends only on $n$, $p$, $\beta$ and the Sobolev constant for $\Omega$.
	
	Let $x_0\in \R^n$ and $\rho>0$ be such that 
	$$
	||u||^{p^\ast-p}_{L^{p^\ast}(B_\rho(x_0) \cap \Omega)}\leq \frac{c_4}2.
	$$
	Let $R < \rho$ and let $\eta$ be such that ${\rm supp}(\eta)\subset B_R(x_0)$.  From Holder's inequality applied to the last term in \eqref{post_sobolev},  we obtain
	\begin{equation}
	\frac{c_4}2\left(\int_{\Omega}F^{p^\ast}(u)\eta^{p^\ast}\, dx\right)^{\frac{p}{p^\ast}}\leq \int_{\Omega}\eta^pu^{p}G'(u)\, dx+\int_{\Omega}(\eta^p+|\nabla\eta|^p)F^p(u)\, dx\, .
	\end{equation}
	By taking the limit as $l\rightarrow\infty$, from the definition of $F$ and $G$ and since $\eta \geq 0$, by monotone convergence we conclude that 
	
	\begin{equation}
	\begin{aligned}
	\frac{c_4}2\left(\int_{\Omega}\eta^{p^\ast}u^{q p^\ast}\, dx\right)^{\frac{p}{p^\ast}}\leq
	\int_{\Omega}\eta^pu^{p}u^{(q-1)p}\, dx + \int_{\Omega}(\eta^p+|\nabla\eta|^p) u^{q p}\, dx  
	\leq 
	2\int_{\Omega}(\eta^p+|\nabla \eta|^p)u^{qp}\, dx.
	\end{aligned}
	\end{equation}
	Hence, if $\rho>R>R'>0$ and we take $\eta\in C_c^\infty(B_R(x_0))$, $0\leq \eta \leq 1$, $\eta=1$ in $B_{R'}(x_0)$, $|\nabla\eta|\leq \frac{1}{R-R'}$, then we have
	%	
	%	$$
	%	\|u^q\|_{p^*, R'}\leq c \left (1+\frac{1}{R-R'} \right) \|u^q\|_{p,R},
	%	$$
	%	or
	
	\begin{equation}
	\label{prima_moser}
	\|u\|_{qp^*, R'}\leq c^\frac{1}{q} \left (1+\frac{1}{R-R'} \right)^\frac{1}{q} \|u\|_{qp,R},
	\end{equation}
	where here and in the following we set 
	$$
	\|u\|_{\alpha, r}:=\left (\int_{\Omega\cap B_r(x_0)} u^\alpha dx\right)^{\frac 1 \alpha}
	$$ 
	and where $c>0$  depends only on $n$, $p$, $\beta$ and the Sobolev constant for $\Omega$. This completes the proof of Step 1.

\medskip	
	
\noindent \emph{Step 2: Moser iteration.} We define $R_j=r(1+2^{-j})$ with $0<r<\rho/2$ and  $j=0,1,2,\dots$ and $q_j=\left (\frac{p^*}{p}\right )^j$ (note that $q_j\geq 1$). We use \eqref{prima_moser} with $R=R_j$, $R'=R_{j+1}$ and $q=q_j$ to obtain 
	
	\begin{equation*}
	\|u\|_{q_jp, R_{j+1}}\leq c^\frac{1}{q_j} \left (1+2^{j+1}r^{-1} \right)^\frac{1}{q_j} \|u\|_{q_jp,R_j}
	%\leq c^\frac{1}{q_j} 2^\frac{j+2}{q_j}r \|u\|_{q_jp,R_j},
	\end{equation*}
	which implies that

	\begin{equation*}
	\|u\|_{q_jp, R_{j+1}}\leq c^{\sum_{k=0}^j \frac{1}{q_k}} \Pi_{k=0}^j \left (1+2^{k+1}r^{-1} \right)^\frac{1}{q_k}  \|u\|_{q'p,2r}
	\end{equation*}
	By taking the limit for $j\to \infty$ and observing that $\|u\|_{\infty, r}\leq \lim_{j\to \infty} \|u\|_{q_jp,R_j}$, we finally have 
	
	\begin{equation}\label{u_infty}
	\|u\|_{\infty, r}\leq K' \|u\|_{p,2r},
	\end{equation}
	with some $K'>0$ depending only on $n$, $\beta$, $p$ and the Sobolev constant for $\Omega$; hence $
	\|u\|_{\infty}\leq K' \|u\|_{p}$. 
	We recall that \eqref{u_infty} holds for $\tilde u$ in place of $u$, where $\tilde u=u+s$. Coming back to the old variable $u$, we immediately have 
	\begin{equation*}
	\|u\|_{\infty}\leq K (\|u\|_{p}+s) \,,
	\end{equation*}	
	for some $K>0$, which completes the proof.
\end{proof}

\subsection{Higher integrability result}
%In this section we show the following theorem
In this subsection we prove a higher integrability result for $a(\nabla u)$. It is well-known that solutions to $p$-Laplace type equation are only $C^{1,\alpha}$ regular (\cite{DiBenedetto} and \cite{Tolksdorf}), and one may expect higher regularity results if one consider the vector field $a(\nabla u)$. Since we work by approximation, we need to find uniform estimates for the operators which approximate $\diver(a(\nabla u))$. In particular, following \cite{AKM}, we will consider approximating operators of the form $\diver (a^s(\nabla u))$, where $a^s$ satisfies the following conditions: $a^s:\R^n \to \R^n$, $a^s\in C^1(\R^n)$ for $p\geq 2$ and $a^s\in C^1(\R^n\backslash\{\mathcal O\})$ for $p< 2$, and for any $\xi,z\in \R^n$
\begin{equation}
\label{hp_Mingione}
\begin{cases}
|a^s(\xi)|+|\partial a^s(\xi)|(|\xi|^2+s^2)\leq L(|\xi|^2+s^2)^{(p-1)/2}\\
\mu (|\xi|^2+s^2)^{(p-1)/2}|z|^2\leq \partial a^s(\xi)z  \cdot z,
\end{cases}
\end{equation}
with $0<\mu \leq L$ and $s\geq 0$. If $p<2$ we also assume that $a^s$ is symmetric ($\partial_i a^s_j=\partial_j a^s_i$ for any $i,j=1,\dots,n$). The idea is that the operators $a^s$ approximate $a$ as $s \to 0$.

\begin{proposition} \label{prop_mingione}
Let $\Omega\subset \mathbb{R}^n$ be a bounded convex domain. 
Let $a$ satisfy \eqref{hp_Mingione}, $f,h\in C^1([0,\infty))$ such that $h'\geq 0$ 
and let $u$ be a positive weak solution of 
\begin{equation} 
\begin{cases}\label{prob_appr}
\diver(a(\nabla u))  + f(u) = 0 & \text{ in }\Omega \,,\\
a(\nabla u ) \cdot \nu + h(u) = 0 & \text{ on } \partial \Omega \,.
\end{cases}
\end{equation}
Then we have
%$$
%a(\nabla u ) \in W^{1,2}_{loc} (\Omega) \,.
%$$
%Moreover, if $\Omega$ is convex and bounded, then
$$
\|a(\nabla u )\|_{W^{1,2}(\Omega)} \leq C \,,
$$
where $C$ depends only on $n$, $p$, $L$, $\mu$, the perimeter of $\Omega$, the $C^1$ norm of $u$ in $\Omega$ and the Lipschitz seminorm of $f$ and $h$ in $[0,\max u]$.
\end{proposition}

\begin{proof}
%\emph{Step 1:} $a(\nabla u ) \in W^{1,2} (\Omega \cap B_R) \,,$ for any $R>0$ 

%This is proved by following \cite{AKM}[Theorem 4.1], see also \cite{CFR}[Proposition 2.2].

%\bigskip

We follow the proof of \cite[Proposition 2.8]{CFR}. \\
We approximate $\Omega$ by a sequence of open convex domains $\{\Omega_k\}$ such that $\Omega_k \subseteq \Omega$ and $\partial \Omega_k$ is smooth. 
Also, we fix a point $\bar x \in \cap_k\Omega_k$, and for $k$ fixed we consider the following problem
\begin{equation}%\label{approx_1bis}
\begin{cases}
\diver(a(\nabla u_k)) + f(u) = 0 & \text{ in } \Omega_k
\\
u_k(\bar x)=u(\bar x)\\
a(\nabla  u_k)\cdot\nu+h(u)=0 & \text{ on } \partial\Omega_k\,,
\end{cases}
\end{equation}
which has a unique solution $u_k$ that can be found by considering first the minimizer $v_{k}$ of the minimization problem
	$$
	\min_v\left\{\int_{\Omega_k}\left[\frac1p H(\nabla v)^p -f(u)v\right]\,dx+\int_{\partial\Omega_k} h(u)v d\sigma\,\right\},
	$$
	then setting $u_{k}(x):=v_{k}(x)+u(\bar x)-v_{k}(\bar x)$, and finally taking the limit of $u_{k}$ (note that the functions $ u_{k}$ are uniformly $C^{1,\theta}$ in every compact subset of $\Omega$, and uniformly H\"older continuous up to the boundary).
Let $\{\phi_l\}$ be a family of radially symmetric smooth mollifiers and define
\begin{equation} \label{ajz}
a^l(z):=(a\ast\phi_l)(z)\qquad \text{for $z\in\mathbb{R}^n$} \, .
\end{equation}
By standard properties of convolution and the continuity of $a(\cdot)$ we have that
$a^l$ converges to $ a $ uniformly on compact subset of $\mathbb{R}^n$. Moreover, $a^l$ satisfies \eqref{hp_Mingione} with $s$ replaced by $s_l$, where $s_l\to 0$ as $l\to \infty$.

Let $u_{k,l}$ be a solution of 
\begin{equation}\label{approx_1}
\begin{cases}
\diver(a^l(\nabla u_{k,l})) + f(u) = 0 & \text{ in } \Omega_k
\\
a^l(\nabla u_{k,l})\cdot\nu+h(u)=0 & \text{ on } \partial\Omega_k \,,
\end{cases}
\end{equation}
which can be constructed analogously to $u_k$. We notice that $u_{k,l}$ is unique up to an additive constant.
Also, because $u$ is locally bounded and $\Omega_k$ is smooth, the functions $u_{k,l}$ are $C^{1,\theta}(\overline\Omega_k) $, uniformly in $l$. In particular, assuming without loss of generality that $u_{k,l}(\bar x)=u(\bar x)$ for some fixed point $\bar x \in \Omega_k$, as $l \to \infty$ one sees that $u_{k,l}$ converges in $C^1_{\rm loc}$ to the unique solution $\bar u_k$ of 
\begin{equation}\label{approx_1bis}
\begin{cases}
\diver(a(\nabla \bar u_k)) + f(u) = 0 & \text{ in } \Omega_k
\\
\bar u_k(\bar x)=u(\bar x)\\
a(\nabla \bar u_k)\cdot\nu+h(u)=0 & \text{ on } \partial\Omega_k \,.
\end{cases}
\end{equation}
The function $u_k$ is also a solution of \eqref{approx_1bis}, so by uniqueness that $\bar u_k=u_k$ and therefore
$u_{k,l}$ converges to $u_k$ as $l \to \infty$. In the same way, $u_k\to u$ as $k\to \infty$.

%\underline{First case:} $s>0$ and $a$ smooth  and $\partial \Omega$ smooth. 

Since $u$ is uniformly positive in $\Omega_k$, then also $u_k,u_{k,l}$ are uniformly positive inside $\Omega_k$. 
% for large $k,l$. In the sequel we shall always assume that $k$ and $l$ are sufficiently large so that this property holds. 
%For simplicity of notation, we shall drop the dependency on $k$ and we write $u_l, \Omega,\Omega_R,\Gamma_0,\Gamma_1$ instead of $u_{k,l}, \Omega_{k},\Omega_{k,R},\Gamma_{k,0},\Gamma_{k,1}$, respectively. \\
We multiply the equation in \eqref{prob_appr} by $\psi\in C^\infty_c(\R^n)$ and integrate over $\Omega_k$ to get
\begin{equation*}\int_{\Omega_k}\diver(a^l(\nabla u_{k,l}))\psi\, dx =-\int_{\Omega_k} f(u)\psi\, dx,
\end{equation*}
that together with the divergence theorem gives
\begin{equation*}\label{previous}
\begin{aligned}
-\int_{\Omega_k} a^l (\nabla u_{k,l} )\cdot\nabla\psi\, dx+\int_{\partial\Omega_k}\psi a^{l} (\nabla u_{k,l} )\cdot\nu\, d\sigma=-\int_{\Omega_k} f(u)\psi\, dx\, .
\end{aligned}
\end{equation*}
%\todo{Since}
%$$
%\int_{\partial\Omega_k}\psi a^l(\nabla u_l)\cdot\nu\, d\sigma=\int_{\Gamma_{1}}\psi a^l(\nabla u_l)\cdot\nu\, d\sigma + \int_{\Gamma_{0}}\psi a^l(\nabla u_l)\cdot\nu\, d\sigma\, ,
%$$
%from the fact that {$\psi\in C^\infty_c(B_R)$} 
Since $u$ satisfies \eqref{approx_1} we obtain that 
\begin{equation}\label{previous_bis}
\begin{aligned}
\int_{\Omega_k} a^{l} (\nabla u_{k,l} )\cdot\nabla\psi\, dx=\int_{\Omega_k} f(u)\psi\, dx- \int_{\partial \Omega_k} h(u) \psi d\sigma\,.
\end{aligned}
\end{equation}
Let  $\varphi \in C^\infty_c(\R^n)$, and for
$\delta>0$ small define the set  
$$
\Omega_{k,\delta}:=\{ x\in\Omega_k \, : \,  \dist(x,\partial\Omega_k)>\delta\}\, .
$$ 
The domain $\Omega_{k}\cap {\rm supp}(\varphi)$ is smooth, then for $\delta$ small enough  
$\Omega_{k,\delta}\setminus\Omega_{k,2\delta}$ is of class $C^{\infty}$ inside the support of $\varphi$.
In particular, every $x \in (\Omega_{k,\delta}\setminus\Omega_{k,2\delta})\cap {\rm supp}(\varphi)$ can be written as 
$$
x=y - |x-y|\nu(y)
$$
where $y=y(x)\in \partial \Omega_{k,\delta}$ is the projection of $x$ on $\partial \Omega_{k,\delta}$ and $\nu(y)$ is the outward normal to $\partial \Omega_{k,\delta}$ at $y$. 
In addition, the set $(\Omega_{k,\delta}\setminus\Omega_{k,2\delta})\cap {\rm supp}(\varphi)$ can be parametrized on $\partial\Omega_{k,\delta}$ by a $C^1$ function $g$ (see \cite[Formula 14.98]{GT}).

Let us consider a cut-off function $\zeta_\delta:\Omega_k \to [0,1]$ satisfying $\zeta_\delta=1$ in $\Omega_{k,2\delta}$, $\zeta_\delta=0$ in $\Omega_k\setminus\Omega_{k,\delta}$, and 
$$
\nabla\zeta_\delta(x)=-\frac{1}{\delta}\nu(y(x))\qquad \text{inside }\Omega_{k,\delta}\setminus\Omega_{k,2\delta}\,.
$$
Using $\psi=\partial_m(\varphi\zeta_\delta)$ in \eqref{previous_bis} with $m\in \{1,\ldots,n\}$ and integrating by parts, we get 
\begin{equation*}
\begin{aligned}
\sum_{i=1}^n\left(\int_{\Omega_k} \partial_m a_i^l(\nabla u_{k,l})\zeta_\delta\partial_i\varphi\, dx+\int_{\Omega_k} \partial_m a_i^l(\nabla u_{k,l})\varphi\partial_i\zeta_\delta\, dx\right)=\int_{\Omega_k} \partial_m(f(u))\varphi\zeta_\delta\, dx\,
%\text{\todo{$+ \int_{\partial \Omega_k} h(u) \partial_m(\varphi\zeta_\delta) d\sigma$}} ,
\end{aligned}
\end{equation*}
where we use the notation $a^{l}=(a_1^l,\ldots,a_n^l)$ to denote the components of the vector field $a^l$. \\
From the definition of $\zeta_\delta$, we clearly have  
\begin{equation*}\label{lim1}
\lim_{\delta\rightarrow 0}\int_{\Omega_k} \partial_m a_i^l(\nabla u_{k,l})\zeta_\delta\partial_i\varphi\, dx=\int_{\Omega_k} \partial_m a_i^l(\nabla u_{k,l})\partial_i\varphi\, dx\, .
\end{equation*}
Also, if we set
$$
w(x)=\partial_m a_i^l(\nabla u(x))\varphi(x)\,,
$$
by the coarea formula we have
\begin{equation}\label{Magna}
\begin{aligned}
\int_{\Omega_{k,\delta}\setminus\Omega_{k,2\delta}} w\partial_i\zeta_\delta\, dx &  = - \frac{1}{\delta} \int_{\Omega_{k,\delta}\setminus\Omega_{k,2\delta}}\nu_{i}(y(x))w dx \\
& =  - \frac{1}{\delta} \int_{\delta}^{2 \delta} dt \int_{\partial\Omega_{k,\delta}} \nu_{i}(y(x))w(y-t\nu(y))|{\rm det}(Dg)|d\sigma(y) \\ & =  -  \int_{1}^{2} ds \int_{\partial\Omega_{k,s\delta}}  w(y - s \delta \nu(y))  \nu_i(y) |{\rm det}(Dg)|d\sigma(y) \, .
\end{aligned}
\end{equation}
Since $w \in C^0$, we can pass to the limit and obtain
\begin{equation*}
\lim_{\delta \to 0} \int_{\Omega_k} \partial_m a_i^l(\nabla u_{k,l})\varphi\partial_i\zeta_\delta\, dx = - \int_{\partial\Omega_k}  \partial_m a_i^l(\nabla u_{k,l})\varphi\nu_i d \sigma \,.
\end{equation*}
Hence, we proved that
\begin{equation}\label{11.16}
\begin{aligned}
\sum_{i=1}^n\left(\int_{\Omega_k} \partial_m a_i^l(\nabla u_{k,l})\partial_i\varphi\, dx- \int_{\partial\Omega_k}  \partial_m a_i^l(\nabla u_{k,l})\varphi\nu_i d \sigma\right)=\int_{\Omega_k} \partial_m(f(u))\varphi\, dx\,
%\text{\todo{$+ \int_{\partial \Omega_R} h(u) \partial_m(\varphi\zeta_\delta) d\sigma$}}  .
\end{aligned}
\end{equation}
Now, 
let 
$$
\Omega_{k,\delta}^{t}:=\{ x\in\Omega_{k,\delta} \, : \, \dist(x,\partial\Omega_{k,\delta})>t\}\, .
$$
If $x \in (\Omega_{k,\delta}\setminus\Omega_{k,2\delta})\cap {\rm supp}(\varphi)$ with $x=y - t \nu(y)$, then $x\in\partial\Omega_{k,\delta}^{t}$ and the outward normal to $\partial\Omega_{k,\delta}^{t}$ at $x$ coincides with the outward normal to $\partial\Omega_{k,\delta}$ at $y$. 
%Hence, by writing $\nu(x)$ in place of $\nu(y)$
Thus, we have
\begin{equation*}\label{int2}
\begin{aligned}
\partial_m a_i^l(\nabla u_{k,l}(x))\varphi(x)\nu_i(x)=&\,\varphi(x)\partial_m(a^l(\nabla u_{k,l}(x))\cdot\nu(x))\\ 
&-\varphi(x)a_i^l(\nabla u_{k,l}(x))\partial_m\nu_i(x)\, .
\end{aligned}
\end{equation*}
Let $m\in\{1,\dots,n\}$. Then, writing $\varphi=a_m^l(\nabla u_{k,l})$, we obtain
\begin{equation}\label{int3}
\begin{aligned}
\partial_m a_i^l(\nabla u_{k,l}(x))\varphi(x)\nu_i(x)=&\,a_m^l(\nabla u_{k,l}(x))\partial_m\bigl(a^l(\nabla u_{k,l}(x))\cdot\nu(x)\bigr)\\ 
&-a_m^l(\nabla u_{k,l}(x))a_i^l(\nabla u_{k,l}(x))\partial_m\nu_i(x) \, .
\end{aligned}
\end{equation}
We notice that $\partial_m\nu_i (x)$ is the second fundamental form $\mathrm{II}_{x}^{t}$ of $\partial\Omega_{k,\delta}^{t}$ at $x$:
$$
\sum_{i,m=1}^n\partial_m\nu_i (x)a_i^l(\nabla u_{k,l}(x))a_m^l(\nabla u_{k,l}(x))=\mathrm{II}_{x}^{t}(A^{k,l}(x)),A^{k,l}(x))\, .
$$
where $A^{k,l}(x)$ is the tangential component of $a^l(\nabla u_{k,l}(x))$.  
Moreover, $\mathrm{II}_{x}^{t}$ is non-negative definite because $\Omega_{k}$ is convex. This means that
\begin{equation}\label{convexity1}
\sum_{i,m=1}^n\partial_m\nu_i (x)a_i^l(\nabla u_{k,l}(x))a_m^l(\nabla u_{k,l}(x))\geq 0\, .
\end{equation}
Hence \eqref{int3} becomes
\begin{align}\label{int4}
\begin{aligned}
\sum_{i,m=1}^n\partial_m a_i^l(\nabla u_{k,l}(x))\varphi(x)\nu_i(x)
&\leq  \sum_{i,m=1}^na_m^l(\nabla u_{l}(x))\partial_m\bigl(a^l(\nabla u_{k,l}(x))\cdot\nu(x)\bigr)\, \\
%&-\mathrm{II}_{x}^{t}(a^l(\nabla u_{k,l}(x)),a^l(\nabla u_{k,l}(x)))\, .
\end{aligned}
\end{align}
Recalling that $\varphi=a_m^l(\nabla u_{k,l})$ and the boundary condition of \eqref{approx_1} we get
\begin{equation*}
\begin{aligned}
\sum_{i,m=1}^n\int_{\partial\Omega_k}  \partial_m a_i^l(\nabla u_{k,l})\varphi\nu_i d \sigma &\leq  \sum_{i,m=1}^n \int_{\partial\Omega_k}a_m^l(\nabla u_{k,l})\partial_m\bigl(a^l(\nabla u_{k,l})\cdot\nu\bigr)\, d\sigma\\
%&-\int_{\partial\Omega_k}\mathrm{II}_{x}^{t}(a^l(\nabla u_{k,l}(x)),a^l(\nabla u_{k,l}(x)))\eta^2(x)\, dx\\
&=n \int_{\partial\Omega_{k}}\,a^l(\nabla u_{k,l})\cdot \nabla \bigl(a^l(\nabla u_{k,l})\cdot\nu\bigr)\, d\sigma \\
%&-\int_{\partial\Omega_k}\mathrm{II}_{x}^{t}(a^l(\nabla u_{k,l}(x)),a^l(\nabla u_{k,l}(x)))\eta^2(x)\, dx\\ 
&=-n \int_{\partial\Omega_{k}}h'(u)\,a^l(\nabla u_{k,l})\cdot \nabla (u) \, dx \\
%&-\int_{\partial\Omega_k}\mathrm{II}_{x}^{t}(a^l(\nabla u_{k,l}(x)),a^l(\nabla u_{k,l}(x)))\eta^2(x)\, dx\,.
\end{aligned}
\end{equation*}
%where the last equality follows from the condition $a^l(\nabla u)\cdot\nu=0$ on $\partial\Omega_{R}$. Indeed, this condition  implies that $a^l(\nabla u)$ is a tangent vector-field and that the tangential derivative of $a^l(\nabla u)\cdot\nu$ vanishes on $\partial\Omega_{R}$.
Thus, by \eqref{11.16} we obtain

\begin{align*}%\label{Jannacci_bis}
\nonumber\sum_{i,m=1}^n\int_{\Omega_k} \partial_m a_i^l(\nabla u_{k,l})\partial_i\left(a_m^l(\nabla  u_{k,l})\right)\, dx \leq 
& n\int_{\Omega_k}  \nabla (f(u))\cdot a^l(\nabla  u_{k,l})\, dx \\
\nonumber &-n \int_{\partial\Omega_k}h'(u)\,a^l(\nabla u_{k,l})\cdot \nabla u \, d\sigma \,.
\end{align*}
%Since $u$ is bounded, $h$ is Locally Lipschitz and $u_{k,l}$ and $u$ are uniformly bounded in $C^{1,\alpha} (\Omega_k)$ (and hence also on $\partial \Omega_k$, then by Holder inequality we obtain
%\begin{align*}%\label{Jannacci_bis}
%\nonumber\sum_{i,m=1}^n\int_{\Omega_k} \partial_m a_i^l(\nabla u_{k,l})\partial_i\left(a_m^l(\nabla  u_{k,l})u_{k,l}^{\gamma}\right)\, dx \leq 
%& n\int_{\Omega_k}  \nabla (f(u))\cdot a^l(\nabla  u_{k,l})u_{k,l}^{\gamma}\, dx \\
%\nonumber &+ C \|u_{k,l}^\gamma \|_{C^0(\Omega_k)}\|u_{k,l}^\gamma \|_{C^1(\Omega_k)}^{p-1} Lip(h,[0,\max u]) \|u\|_{C^1(\Omega_k)} |\partial\Omega_k| \,.
%\end{align*}

Now, proceeding as in \cite[Proposition 2.8]{CFR} we get 
\begin{align} 
\label{stima1}
\nonumber \int_{\Omega_k}|\nabla(a^l(\nabla u_{k,l}))|^2\, dx&\leq C\int_{\Omega_k}|a^l(\nabla u_{k,l})|^2\, dx+C \int_{\Omega_k} |\nabla (f(u))||a^l(\nabla  u_{k,l})|\, dx\\
 &-C \int_{\partial\Omega_k}h'(u)\,a^l(\nabla u_{k,l})\cdot \nabla (u) \, d\sigma.
\end{align}
The right hand side of \eqref{stima1} is uniformly bounded by a number depending only by the following quantities $\|u\|_{C^1(\Omega_k)}, |\partial\Omega_k|$ and by the Lipschitz constants of $f$ and $h$ on $[0,\max u]$.  Thus
%, since $u_{k,l} \to u_k$ in $C^1_{\rm loc}$ and $a^l\to a$ locally uniformly, 
we can pass to the limit first for $l\to \infty$ and then for $k\to \infty$
to deduce that
%\begin{align} \label{robertocarlos}
%\nonumber\int_{\Omega_{k}}|\nabla(a(\nabla u_k))|^2u_k^{\gamma}\, dx&\leq C\int_{\Omega_{k}}|a(\nabla u_k)|^2|\nabla( u_k^{\frac{\gamma}{2}})|^2\, dx+C \int_{\Omega_{k}} |\nabla (f(u))||a(\nabla  u_k)|u_k^{\gamma}\, dx\\
% &-C \int_{\partial\Omega_{k}}u_k^{\gamma}h'(u)\,a(\nabla u_k)\cdot \nabla (u) \, d\sigma 
%%&-C\int_{\partial\Omega_{k,R}}\mathrm{II}_{x}^{t}(a(\nabla u_k(x)),a(\nabla u_k(x)))u_k^{\gamma}(x)\eta^2(x)\, dx. 
%\end{align}
%
%In particular, taking $\gamma=0$, \eqref{robertocarlos} proves that $a(\nabla u_k)\in W^{1,2}_{\rm loc}(\overline{\Omega_{k}})$, and $\{a(\nabla u_k)\}_{k \in \N}$ is uniformly bounded in $W^{1,2}_{\rm loc}$. Hence, letting $k \to\infty$ in \eqref{robertocarlos} we obtain 
\begin{align*} 
\int_{\Omega}|\nabla(a(\nabla u))|^2\, dx&\leq C\int_{\Omega}|a(\nabla u)|^2\, dx+C \int_{\Omega} |\nabla (f(u))||a(\nabla  u)|\, dx\\
\nonumber &-C \int_{\partial\Omega}h'(u)\,H^p(\nabla u) \, d\sigma \,.
%&-C\int_{\partial(\Omega\cap B_R)}\mathrm{II}_{x}^{t}(a(\nabla u(x)),a(\nabla u(x)))u^{\gamma}(x)\eta^2(x)\, dx. 
\end{align*}
Since $h' \geq 0$, we obtain
\begin{equation*} 
\int_{\Omega}|\nabla(a(\nabla u))|^2\, dx\leq C\int_{\Omega}|a(\nabla u)|^2\, dx+C \int_{\Omega} |\nabla (f(u))||a(\nabla  u)|\, dx \,,
\end{equation*}
which proves that $a(\nabla u ) \in W^{1,2}(\Omega)$.
\end{proof}

\section{Proof of Proposition \ref{prop_int_ineq}} \label{section_prop}
In this section we prove Proposition \ref{prop_int_ineq}, which is the main tool to prove Theorems \ref{thm_main} and \ref{thm_main1}. Our approach is based on a differential identity which involves the second elementary symmetric function $S^2(Q)$ of a $n\times n$ matrix $Q$, i.e. the sum of all the principal minors of $Q$ of order two:
\begin{equation}\label{eq200}
S^2(Q)=\frac12\sum_{i,j}S^2_{ij}(Q)q_{ij}\,,
\end{equation}
where 
$$
S^2_{ij}(Q)=-q_{ji} + \delta_{ij} \tr (Q) \,.
$$
We will use the following Newton's type inequality
\begin{equation}\label{newtonIneq}
S^2{(Q)}\le\frac{n-1}{2n}\tr(Q)^2\, .
\end{equation}
which holds for any matrix  $Q=AB$, where $A$ and $B$ are two symmetric $n\times n$ matrices with $A$ positive semidefinite (see \cite[Lemma 3.2]{CS}).

Moreover, if ${\tr} (Q)\neq 0$ and equality holds in (\ref{newtonIneq}), then 
\begin{equation*}
Q=\frac{{\tr}(Q)}{n}\, {\rm Id} \,,
\end{equation*}
and $A$ is positive definite.

\medskip 

Before giving the proof of Proposition \ref{prop_int_ineq}, we need some preliminary manipulation of the solution and two preliminary lemmas. Let $u$ be a bounded solution of \eqref{pb_main2}, i.e. of
\begin{equation*} %\label{pb_general}
\begin{cases}
\diver (a(\nabla u))  + f(u) = 0 & \text{ in }\Omega \,,\\
a(\nabla u) \cdot \nu +h(u) = 0 & \text{ on } \partial \Omega \,,
\end{cases}
\end{equation*}
where
$$
a(\nabla u) = H^{p-1}(\nabla u)\nabla H(\nabla u)\,,
$$
and we recall that we set
$$
\Delta_p^H u := \diver (a(\nabla u)) \,.
$$
We consider the auxiliary function $v$, defined in the following way
\begin{equation}
\label{def_v}
u=v^{-\frac{n-p}{p}}\, ,
\end{equation}
and we set 
\begin{equation} \label{H_hat}
\hat H(\xi )= H(-\xi) \,.
\end{equation}
It is straightforward to verify that $v$ is a solution to
\begin{equation}\label{eq_v}
\Delta^{\hat H}_p v=
%\left(\dfrac{p}{n-p}\right)^{p-1}f(v^{-\frac{n-p}{p}})v^{\frac{n(p-1)}{p}}
\hat f(v)+\dfrac{n(p-1)}{p}\dfrac{\hat H^p(\nabla v)}{v}\, 
\end{equation}
with boundary condition 
\begin{equation} \label{BC_Robin_v}
\hat a(\nabla v) \cdot \nu - \hat h(v) = 0 \quad \text{ on } \partial \Omega \,,
\end{equation}
where we set
$$
\hat a (\xi)=\hat H^{p-1}(\xi)\nabla \hat H(\xi),
$$
\begin{equation} \label{hat_f_v}
\hat f (v) = \left(\dfrac{p}{n-p}\right)^{p-1}f(v^{-\frac{n-p}{p}})v^{\frac{n(p-1)}{p}}\,,
\end{equation}
$$
 \hat h(v) = c_p v^{\frac{n}{p}  (p-1) } h(v^{-\frac{n}{p}+1}),
$$
and 
$$
c_p=\left ( \frac{p}{n - p}  \right)^{p-1}.
$$
%Indeed,  
%$$
%u= v^{-\frac{n}{p}+1} \,, \quad \nabla u = \left (-\frac{n}{p} + 1\right ) v^{-\frac{n}{p} } \nabla v \,,
%$$
%then the boundary condition can be written as
%$$
%-\left(\frac{n}{p}-1\right )^{p-1} v^{-\frac{n}{p} (p-1)} a(\nabla v) \cdot \nu + h(v^{-\frac{n}{p} + 1}) = 0\,,
%$$
%i.e.
%\begin{equation*}
%a(\nabla v) \cdot \nu - c_p v^{\frac{n}{p} (p-1) } h(v^{-\frac{n}{p} +1}) = 0 \quad \text{ on } \partial \Omega \,,
%\end{equation*}
%which is \eqref{BC_Robin_v}. 
Moreover, we set
\begin{equation}
\label{def_VW}
V(\xi)=\dfrac{\hat H^p(\xi)}{p},\, \xi\in\mathbb{R}^n, \,
\text{ and }\, W=\nabla[\nabla_\xi V(\nabla v)]=V_{\xi_i\xi_j}(\nabla v)v_{ij} \,.
\end{equation}
We notice that $W=\nabla \hat a(\nabla v)$, i.e. $W=(w_{ij})$ with $w_{ij}=\partial_j \hat a^i(\nabla v)$.

It is clear that $v$ inherits some regularity properties from $u$. In particular, $v$ is $C^{1,\alpha}$ regular and, since 
$$
\hat a(\nabla v) = -\left(\dfrac{p}{n-p}\right)^{p-1}u^{-\frac{n(p-1)}{n-p}}a(\nabla u)
$$
\begin{equation*}
w_{ij} =-\left(\dfrac{p}{n-p}\right)^{p-1}\left[u^{-\frac{n(p-1)}{n-p}}\partial_j a^i(\nabla u) - \dfrac{n(p-1)}{n-p}u^{\frac{p(1-n)}{n-p}}u_i\, a^j (\nabla u)\right] \,,
\end{equation*}
from Proposition \ref{prop_mingione} and since $u$ is $C^{1,\alpha}$ regular we have that $\hat a(\nabla v) \in W^{1,2}_{loc}(\Omega)$ (see also \cite[Lemma 3.1]{CFR}).

\medskip

The starting point is the following integral identity which was proved in \cite[Lemma 3.3]{CFR}.

\begin{lemma}[{\cite[Lemma 3.3]{CFR}}] \label{lemma_integr1}
	Let $v$ be given by \eqref{def_v}, let $V$ and $W$ be as in \eqref{def_VW}.  Then, for any $\varphi\in C^{\infty}_c(\Omega)$, we have 
	\begin{equation}\label{BC_integral}
	\begin{aligned}
	\int_{\Omega}&\left(2v^{1-n} S^2(W)+(n-1) np(p-1)v^{-n-1}V^2(\nabla v)+(1-n)(2p-1)v^{-n}V(\nabla v)\Delta^{\hat H}_p v\right) \varphi \, dx\\=&-\int_{\Omega}\varphi_{j}\bigl(v^{1-n} S_{ij}^2(W)V_{\xi_i}(\nabla v)+(1-n)(p-1)v^{-n}V(\nabla v)V_{\xi_j}(\nabla v)\bigr)\, dx \, .
	\end{aligned}
	\end{equation}
%	\todo{****}
%	\begin{multline} \label{mer22gen}
%	\int_{\Omega}\left(\dfrac{n-1}{n}v^{1-n} (\Delta^H_p v)^2+(n-1)np(p-1)v^{-n-1}V^2(\nabla v)+(1-n)(2p-1)v^{-n}V(\nabla v)\Delta^H_p v\right) \\
%	= \int_{\partial \Omega} \bigl(v^{1-n} S_{ij}^2(W)V_{\xi_i}(\nabla v)+(1-n)(p-1)v^{-n}V(\nabla v)\nabla_\xi V(\nabla v)\bigr) \cdot \nu \, .
%	\end{multline}
\end{lemma}

\begin{proof}
The identity \eqref{BC_integral} follows from the following differential identity which was proved in \cite[Lemma 4.1]{BiCi}:
\begin{equation}\label{BC_anis}
	\begin{aligned}
	2v^{1-n} S^2(W)+&(n-1)np(p-1)v^{-n-1}V^2(\nabla v)+(1-n)(2p-1)v^{-n}V(\nabla v)\Delta^{\hat H}_p v \, ,\\
	& = \,\diver\bigl(v^{1-n} S_{ij}^2(W)V_{\xi_i}(\nabla v)+(1-n)(p-1)v^{-n}V(\nabla v)\nabla_\xi V(\nabla v)\bigr) \,. 
	\end{aligned}
	\end{equation}
This identity holds pointwise for smooth functions $v$ and $V$, and \eqref{BC_integral} is its integral counterpart which can be proved by approximation. Since the proof is the same as the one for \cite[Lemma 3.3]{CFR}, we omit the proof.
\end{proof}

%From Lemma \ref{lemma_integr1} and by using the boundary condition \eqref{BC_Robin_v}, we obtain the following result.

In order to state the next lemma, we consider $\Omega$ a bounded smooth domain and we recall some notation that we used in the proof of Proposition \ref{prop_mingione}. In particular, for $\delta>0$ small we define 
$$
\Omega_{\delta}:=\{ x\in\Omega \, : \,  \dist(x,\partial\Omega)>\delta\}\, .
$$ 
We may and do suppose that for $\delta$ small enough $\Omega_{\delta}\setminus\Omega_{2\delta}$ is of class $C^2$ (see the proof of Proposition \ref{prop_mingione} for the details). In particular, every $x \in \Omega_{\delta}\setminus\Omega_{2\delta}$ can be written as 
\begin{equation}\label{ptox}
x=y - |x-y|\nu(y)
\end{equation}
where $y=y(x)\in \partial \Omega_{\delta}$ is the projection of $x$ on $\partial \Omega_{\delta}$ and $\nu(y)$ is the outward normal to $\partial \Omega_{\delta}$ at $y$. Now we consider a cut-off function $\zeta_\delta:\Omega \to [0,1]$ satisfying $\zeta_\delta=1$ in $\Omega_{2\delta}$, $\zeta_\delta=0$ in $\Omega\setminus\Omega_{\delta}$, and 
$$
\nabla\zeta_\delta(x)=-\frac{1}{\delta}\nu(y(x))\qquad \text{inside }\Omega_{\delta}\setminus\Omega_{2\delta}\,.
$$
With these notations and from Lemma \ref{lemma_integr1}, we obtain the following result.

\begin{lemma} \label{lemma_boundary}
Let $\Omega$ be a bounded smooth domain and let $v$ be given by \eqref{def_v}. Let $V$ and $W$ be as in \eqref{def_VW}.  Then,
%\begin{multline}\label{lem3.2}
%\int_{\Omega}\left(2v^{1-n} S^2(W)+(n-1) np(p-1)v^{-n-1}V^2(\nabla v)+(1-n)(2p-1)v^{-n}V(\nabla v)\Delta^{H}_p v\right)\zeta_\delta =\\ \int_{\partial\Omega_\delta\setminus\Omega_{2\delta}} \left (H(\nabla v)^p \left(- \hat h'(v)  +\frac{p-1}{p} \frac{\hat h(v)}{v} \right) + \hat h (v) \hat f(v) + \mathrm{II}(a^T(\nabla v),a^T(\nabla v))  \right )v^{1-n}\,.
%\end{multline}
%
%\todo{dovrebbe essere...}

\begin{multline}\label{lem3.2}
\int_{\Omega}\left(2v^{1-n} S^2(W)+(n-1) np(p-1)v^{-n-1}V^2(\nabla v)+(1-n)(2p-1)v^{-n}V(\nabla v)\Delta^{\hat H}_p v\right)\zeta_\delta\, dx =\\ \frac{1}{\delta}\int_{\Omega_\delta\setminus\Omega_{2\delta}} \left(- \nabla (\hat a(\nabla v)\cdot\nu) \cdot \hat a(\nabla v) +  \hat a(\nabla v)\cdot\nu\left(\frac{p-1}{p} \frac{\hat H^p(\nabla v)}{v} + \hat f(v)\right) + \mathrm{II}(\hat a^T(\nabla v),\hat a^T(\nabla v))\right)v^{1-n}\, dx\,.
\end{multline}

\end{lemma}

\begin{proof}
The proof of this lemma is based on an approximation argument which is very similar to the one done in the proof of Proposition \ref{prop_mingione} (see also \cite[Proposition 2.8 and Lemma 3.4]{CFR}). Let $\zeta_\delta$ be the already cited cut-off function. By choosing $\varphi=\zeta_\delta$ in the statement of Lemma \ref{lemma_integr1}, we obtain 
%$\delta\to 0$, we obtain 
%\begin{equation}\label{BC_integral_3}
%\int_{\Omega } F \eta  = - \int_{\Omega }\nabla \eta \cdot L + \int_{\partial \Omega} \eta L \cdot \nu d \sigma \,,
%\end{equation}
%where 
\begin{multline}\label{F}
\int_{\Omega}\left(2v^{1-n} S^2(W)+(n-1) np(p-1)v^{-n-1}V^2(\nabla v)+(1-n)(2p-1)v^{-n}V(\nabla v)\Delta^{\hat H}_p v\right) \zeta_\delta\, dx \\=
\frac{1}{\delta}\int_{\Omega_\delta\setminus\Omega_{2\delta}} (v^{1-n} S^2_{ij}(W) V_{\xi_i}(\nabla v) + (1-n) (p-1) v^{-n} V(\nabla v)  V_{\xi_j}(\nabla v)) \nu_j\, dx \, ,
%F= 2v^{1-n} S^2(W)+(1-n)(-n)p(p-1)v^{-n-1}V^2(\nabla v)+(1-n)(2p-1)v^{-n}V(\nabla v)\Delta^{H}_p 
\end{multline}
%and $L=(L_1,\ldots,L_n)$ with
%$$
%L_j = v^{1-n} S_{ij}^2(W)V_{\xi_i}(\nabla v)+(1-n)(p-1)v^{-n}V(\nabla v)V_{\xi_j}(\nabla v)
%$$
%for $j=1,\ldots , n$. 
where $\nu=\nu(y(x))$ is defined as in \eqref{ptox}.

From \eqref{F} it is clear that we have to evaluate the quantity 
\begin{equation} \label{Theta}
\Theta = v^{1-n} S^2_{ij}(W) V_{\xi_i}(\nabla v) \cdot \nu + (1-n) (p-1) v^{-n} V(\nabla v) \nabla_\xi V(\nabla v) \cdot \nu 
\end{equation}
in $\Omega_\delta\setminus\Omega_{2\delta}$. Since $V(\xi)=\hat H(\xi)^p/p$, $\nabla_\xi V (\xi) = a(\xi) $ and $w_{ji} = \partial_i a_j(\nabla v)$, from the definition of 
$$
S^2_{ij}(W)= -w_{ji} + \delta_{ij} \trace(W) = -w_{ji} + \Delta_p^{\hat H} v \,,
$$ 
we have
\begin{equation*}
\begin{aligned}
S^2_{ij}(W) V_{\xi_i}(\nabla v) \cdot \nu & = -w_{ji} V_{\xi_i} (\nabla v) \nu_j + \Delta_p^{\hat H} v \nabla_\xi V(\nabla v) \cdot \nu \\
& = - \partial_i \hat a_j(\nabla v) \hat a_i(\nabla v) \nu_j + \Delta_p^{\hat H} v \hat a(\nabla v) \cdot \nu \\
& = - \partial_i (\hat a(\nabla v) \cdot \nu)  \hat a_i (\nabla v) + \hat a_i(\nabla v) \hat a_j(\nabla v) \partial_i \nu_j + \Delta_p^{\hat H} v \hat a(\nabla v) \cdot \nu \\
& = - \nabla (\hat a(\nabla v)\cdot\nu) \cdot \hat a(\nabla v) +\hat a(\nabla v)\cdot\nu \Delta_p^{\hat H} v + \mathrm{II}_{x}(\hat a^T(\nabla v),\hat a^T(\nabla v)) \,,
\end{aligned}
\end{equation*}
where $\mathrm{II}_x$ and $\hat a^T(\nabla v)$ are defined as follows: if $x\in\partial\Omega_t$, with $t\in(\delta,2\delta)$, then $\mathrm{II}_{x}$ is the second fundamental form of $\partial\Omega_t$ evaluated at $x$ and $\hat a^T(\nabla v)$ is the tangential component of $\hat a(\nabla v(x))$ with respect to the tangent hyperplane of $\partial\Omega_t$ at $x$.

From \eqref{eq_v}, we obtain that
\begin{multline} \label{manana}
S^2_{ij}(W) V_{\xi_i}(\nabla v) \cdot \nu =  - \nabla (\hat a(\nabla v)\cdot\nu) \cdot\hat a(\nabla v) + \hat a(\nabla v)\cdot\nu\left(\frac{n}{p} (p-1) \frac{\hat H^p(\nabla v)}{v} + \hat f(v)\right)\\ + \mathrm{II}(\hat a^T(\nabla v),\hat a^T(\nabla v))  \,.
\end{multline}
Hence from \eqref{Theta} and \eqref{manana} we obtain 
%\begin{equation*}
%\Theta v^{-(1-n)} \geq H(\nabla v)^p \left(- \hat h'(v)  + \frac{n}{p} (p-1) \frac{\hat h(v)}{v} \right) + \hat h (v) \hat f(v)  + (1-n) \frac{p-1}{p} v^{-1} H(\nabla v)^p \hat h(v) \\
%&+ \mathrm{II}_{x}(A^T(x)),A^T(x)) \,. 
%\end{equation*}
%we obtain 
$$
\Theta v^{n-1} = - \nabla (\hat a(\nabla v)\cdot\nu) \cdot \hat a(\nabla v) +  \hat a(\nabla v)\cdot\nu\left(\frac{p-1}{p} \frac{H^p(\nabla v)}{v} + \hat f(v)\right)\\ + \mathrm{II}(\hat a^T(\nabla v),\hat a^T(\nabla v))  \,,
$$
and then \eqref{lem3.2}.
%If we assume that 
%\begin{equation} \label{h_hp1}
%- \hat h'(v)  +\frac{p-1}{p} \frac{\hat h(v)}{v} \geq 0 
%\end{equation}
%then
%\begin{equation} \label{Theta_geq}
%\Theta  \geq \hat h (v) \hat f(v)  v^{1-n}  \,. 
%\end{equation}
\end{proof}

Now, we are ready to prove the main result of this section.

\begin{proof}[Proof of Proposition \ref{prop_int_ineq}]
	From Lemma \ref{lemma_boundary} and by applying Newton's inequality \eqref{newtonIneq} to $Q=W$, we obtain that 
	\begin{multline}\label{17.42}
	\int_{\Omega}\left(\dfrac{n-1}{n}v^{1-n} (\Delta^{\hat H}_p v)^2+n(n-1) p(p-1)v^{-1-n}V^2(\nabla v)+(1-n)(2p-1)v^{-n}V(\nabla v)\Delta^{\hat H}_p v\right)\zeta_\delta\, dx \\ 
	\geq \frac{1}{\delta}\int_{\Omega_\delta\setminus\Omega_{2\delta}} \left(- \nabla (\hat a(\nabla v)\cdot\nu) \cdot \hat a(\nabla v) +  \hat a(\nabla v)\cdot\nu\left(\frac{p-1}{p} \frac{{\hat H}^p(\nabla v)}{v} + \hat f(v)\right) + \mathrm{II}(\hat a^T(\nabla v),\hat a^T(\nabla v))\right)v^{1-n}\, dx \,.
	\end{multline}
and we notice that the equality is attained if and only if $W$ is a multiple of the identity, i.e. $W=\lambda(x) {\rm Id}$. We will come back later on the equality case, in order to characterize those $v$ such that the equality in \eqref{17.42} is attained.
	  
	From \eqref{eq_v} and after some tedious computation, we obtain that 
	
%\begin{multline}\label{Irene_II}
%	\dfrac{n-1}{n}c_p^2 \int_{\Omega}v^{n+1-2\frac{n}{p}}f^2(v^{-\frac{n-p}{p}})-\dfrac{n-1}{p} c_p \int_{\Omega}v^{-\frac{n}{p}}f(v^{-\frac{n-p}{p}})H^p(\nabla v) \\ 
%	\geq  \frac{c_p}{p} \int_{\partial \Omega} v^{1- 2\frac{n}{p}} H^p(\nabla v) \left( (n-p) h'(v^{-\frac{n-p}{p}}) - (p-1)(n-1) h(v^{-\frac{n-p}{p}}) v^{\frac{n-p}{p}} \right) \\    +  c_p^2 \int_{\partial \Omega} h (v^{-\frac{n-p}{p}})  f(v^{-\frac{n-p}{p}})  v^{n+1-2\frac{n}{p}} 
%	+ \int_{\partial \Omega} v^{1-n} \mathrm{II}(a^T(\nabla v),a^T(\nabla v))	\, .
%\end{multline}
%
%\todo{diventa}

\begin{multline}\label{Irene_II}
	\dfrac{n-1}{n}c_p^2 \int_{\Omega}v^{n+1-2\frac{n}{p}}f^2(v^{-\frac{n-p}{p}})\zeta_\delta\, dx-\dfrac{n-1}{p} c_p \int_{\Omega}v^{-\frac{n}{p}}f(v^{-\frac{n-p}{p}})\hat H^p(\nabla v)\zeta_\delta\, dx \\ 
	\geq \frac{1}{\delta}\int_{\Omega_\delta\setminus\Omega_{2\delta}} \left[-v^{1-n}\nabla (\hat a(\nabla v)\cdot\nu) \cdot\hat a(\nabla v) + v^{1-n} \hat  a(\nabla v)\cdot\nu \frac{p-1}{p} \frac{\hat H^p(\nabla v)}{v}\right]\, dx \\
	+\frac{1}{\delta}\int_{\Omega_\delta\setminus\Omega_{2\delta}} v^{1-n}\hat a(\nabla v)\cdot\nu \hat f(v)\, dx+ \frac{1}{\delta}\int_{\Omega_\delta\setminus\Omega_{2\delta}} v^{1-n} \mathrm{II}(\hat a^T(\nabla v),\hat a^T(\nabla v))\, dx	\, .
\end{multline}

We write \eqref{Irene_II} in terms of $u$ by using \eqref{def_v} and the homogeneity properties of $H$ and, by dividing by $c_p^2$, we get

%\begin{multline}\label{Irene_u}
%	\dfrac{n-1}{n}\int_{\Omega}u^{\frac{n}{n-p} - p^* +1}f^2(u)-\dfrac{n-1}{n-p} \int_{\Omega}u^{\frac{n}{n-p}- p^*}f(u)H^p(\nabla u) \\ 
%	\geq \frac{1}{n-p} \int_{\partial \Omega} u^{\frac{n}{n-p} - p^* +1} H^p(\nabla u) \left( (n-p) h'(u) - (p-1)(n-1) h(u) u^{-1} \right) \\    +  \int_{\partial \Omega} h (u)  f(u)  u^{\frac{n}{n-p} - p^* +1}
%	+ \int_{\partial \Omega} u^{\frac{n}{n-p} - p^* +1} \mathrm{II}(a^T(\nabla u),a^T(\nabla u))	\, .
%\end{multline}
%
%\todo{diventa}

\begin{multline}\label{Irene_u}
	\dfrac{n-1}{n}\int_{\Omega}u^{\frac{n}{n-p} - p^* +1}f^2(u)\zeta_\delta\, dx-\dfrac{n-1}{n-p} \int_{\Omega}u^{\frac{n}{n-p}- p^*}f(u)H^p(\nabla u)\zeta_\delta\, dx \\ 
	\geq \frac{1}{\delta(n-p)}\int_{\Omega_\delta\setminus\Omega_{2\delta}} u^{\frac{n}{n-p} - p^* +1}  \left( (n-1)(p-1)\frac{H^p(\nabla u)}{u}a(\nabla u)\cdot\nu-(n-p)\nabla (a(\nabla u\cdot\nu)\cdot a(\nabla u)\right)\, dx \\    -  \frac{1}{\delta}\int_{\Omega_\delta\setminus\Omega_{2\delta}} a(\nabla u)\cdot\nu  f(u)  u^{\frac{n}{n-p} - p^* +1}\, dx
	+ \frac{1}{\delta}\int_{\Omega_\delta\setminus\Omega_{2\delta}} u^{\frac{n}{n-p} - p^* +1} \mathrm{II}(a^T(\nabla u),a^T(\nabla u))\, dx	\, .
\end{multline}

Now we simplify \eqref{Irene_u} by using the equation again. We multiply the equation in \eqref{pb_main} by 
$$
u^{\frac{n}{n-p} - p^* +1} f(u)\zeta_\delta
$$ 
and we integrate:
\begin{equation}\label{newnew}
\int_\Omega u^{\frac{n}{n-p} - p^* +1} f(u)\zeta_\delta \diver a( \nabla u)\, dx + \int_\Omega u^{\frac{n}{n-p} - p^* +1} f^2(u)\zeta_\delta\, dx = 0 \,.
\end{equation}
Since $\zeta_\delta$ has compact support in $\Omega$, from the divergence theorem we get
%$$
%\int_\Omega u^{\frac{n}{n-p} - p^* +1} f(u)\zeta_\delta \diver a( \nabla u) = -\int_{\Omega}%\nabla(u^{\frac{n}{n-p} - p^* +1} f(u)\zeta_\delta)\cdot a(\nabla u)+ \int_{\partial\Omega}%u^{\frac{n}{n-p} - p^* +1} f(u)\zeta_\delta a(\nabla u)\cdot\nu 
%$$
%where the integral on the boundary is zero. So coming back to \eqref{newnew} and by %multplying by $(n-1)/n$ we obtain
\begin{multline} \label{Alberto}
\frac{n-1}{n} \int_\Omega u^{\frac{n}{n-p} - p^* +1} f^2(u)\zeta_\delta\, dx - \frac{n-1}{n-p} \int_\Omega u^{\frac{n}{n-p} - p^*} f(u) H^p(\nabla u )\zeta_\delta \, dx \\
= -\frac{n-1}{n\delta} \int_{\Omega_\delta\setminus\Omega_{2\delta}} u^{\frac{n}{n-p} - p^* +1} f(u) a(\nabla u)\cdot\nu\, dx + \frac{n-1}{n} \int_\Omega u^{\frac{n}{n-p}} H^p(\nabla u) \Phi'(u)\zeta_\delta\, dx \,,
\end{multline}
where we recall that 
$$
\Phi(t) = \frac{f(t)}{t^{p^*-1}} \,.
$$
From \eqref{Irene_u} and \eqref{Alberto} we obtain 
\begin{multline*}
\frac{n-1}{n} \int_\Omega u^{\frac{n}{n-p}} H^p(\nabla u) \Phi'(u)\zeta_\delta\, dx \geq  -\frac{1}{n\delta} \int_{\Omega_\delta\setminus\Omega_{2\delta}} u^{\frac{n}{n-p} - p^* +1} f(u) a(\nabla u)\cdot\nu\, dx \\ 
+\frac{1}{\delta(n-p)}\int_{\Omega_\delta\setminus\Omega_{2\delta}} u^{\frac{n}{n-p} - p^* +1}  \left( (n-1)(p-1)\frac{H^p(\nabla u)}{u}a(\nabla u)\cdot\nu-(n-p)\nabla (a(\nabla u)\cdot\nu)\cdot a(\nabla u)\right)\, dx \\ 
+  \frac{1}{\delta}\int_{\Omega_\delta\setminus\Omega_{2\delta}} u^{\frac{n}{n-p} - p^* +1} \mathrm{II}(a^T(\nabla u),a^T(\nabla u))\, dx \,.
\end{multline*}
By taking the limit as $\delta\rightarrow 0$ (as we did in \eqref{Magna}) and by using the boundary condition $a(\nabla u)\cdot\nu+h(u)=0$ on $\partial\Omega$ we get 
\begin{multline*}
\frac{n-1}{n} \int_\Omega u^{\frac{n}{n-p}} H^p(\nabla u) \Phi'(u)\, dx \geq  \frac{1}{n} \int_{\partial \Omega} u^{\frac{n}{n-p} - p^* +1} f(u) h(u)\, d\sigma \\ 
+ \frac{1}{n-p} \int_{\partial \Omega} u^{\frac{n}{n-p} - p^* +1} H^p(\nabla u) \left( (n-p) h'(u) - (p-1)(n-1)\frac{h(u)}{u} \right)\, d\sigma \\ 
+  \int_{\partial \Omega} u^{\frac{n}{n-p} - p^* +1} \mathrm{II}(a^T(\nabla u),a^T(\nabla u))\, d\sigma \,,
\end{multline*}
%\todo{vecchi conti inizio}
%
%From the divergence theorem, by using the boundary condition and by multiplying by $(n-1)/n$ we get
%\begin{multline} \label{Alberto}
%\frac{n-1}{n} \int_\Omega u^{\frac{n}{n-p} - p^* +1} f(u)^2 - \frac{n-1}{n-p} \int_\Omega u^{\frac{n}{n-p} - p^*} f(u) H^p(\nabla u ) \\
%= \frac{n-1}{n} \int_{\partial \Omega} u^{\frac{n}{n-p} - p^* +1} f(u) h(u) + \frac{n-1}{n} \int_\Omega u^{\frac{n}{n-p}} H^p(\nabla u) \Phi'(u) \,,
%\end{multline}
%where we recall that 
%$$
%\Phi(t) = \frac{f(t)}{t^{p^*-1}} \,.
%$$
%From \eqref{Irene_u} and \eqref{Alberto} we obtain 
%\begin{multline*}
%\frac{n-1}{n} \int_\Omega u^{\frac{n}{n-p}} H^p(\nabla u) \Phi'(u) \geq  \frac{1}{n} \int_{\partial \Omega} u^{\frac{n}{n-p} - p^* +1} f(u) h(u) \\ 
%+ \frac{1}{n-p} \int_{\partial \Omega} u^{\frac{n}{n-p} - p^* +1} H^p(\nabla u) \left( (n-p) h'(u) - (p-1)(n-1)\frac{h(u)}{u} \right) \\ 
%+  \int_{\partial \Omega} u^{\frac{n}{n-p} - p^* +1} \mathrm{II}(a^T(\nabla u),a^T(\nabla u)) \,,
%\end{multline*}
%
%\todo{vecchio conti fine}
which is \eqref{integral_ineq}. It remains to consider the equality case in \eqref{17.42}, i.e. when $W=\lambda(x) {\rm Id}$, for some function $\lambda (x)$. Notice that $\Delta_p^{\hat H} v = \tr (W)$ and from \eqref{eq_v} we have that 
$$
\lambda(x)= \frac{1}{n}\Delta^{\hat H}_p v(x)=\frac{1}{n} \hat f(v)+\dfrac{(p-1)}{p}\dfrac{\hat H^p(\nabla v)}{v}\,.
$$
We recall that $v \in C^{1,\alpha}_{loc}$ in $\Omega$ and $v \in C^{2,\alpha}_{loc}$ in $\Omega \setminus Z $, where $Z = \{ \nabla v = 0\}$. Since $f \in C^1$ then $\lambda \in C^{\alpha}_{loc}(\Omega) \cap  C^1_{loc}(\Omega \setminus Z)$. From the definition of $W$ we have that
$$
\partial_i(\hat a_j(\nabla v(x)))=\lambda(x)\delta_{ij}
$$
which implies that $\hat a(\nabla v)\in C^{1,\alpha}_{loc}(\Omega) \cap  C^2_{loc}(\Omega \setminus Z)$. 

For $i \in  \{1,\ldots,n\}$ and choosing $j \neq i$ we find
$$
\partial_i\lambda(x) = \partial_i\bigl(\partial_j(\hat a_j(\nabla v(x)))\bigr)=\partial_j\bigl(\partial_i(\hat a_j(\nabla v(x)))\bigr)=0
$$
for any $x \in \Omega \setminus Z$, which gives that $\lambda$ is constant on each connected component of $ \Omega \setminus Z$. This implies that  
$$
\nabla[\hat a(\nabla v(x))]=W(x)=\lambda \, {\rm Id}
$$ 
in each connected component of $\{\nabla v \neq 0\}$,
i.e. 
$$
\hat a(\nabla v(x))=\lambda (x-x_0)
$$ 
for some $x_0\in\overline{\Omega}$. Hence for any connected component of $\{\nabla v \neq 0\}$ there exist two constants $c_1$ and $c_1$ such that
$$
v(x)=c_1+c_2 \hat H_0(x-x_0)^{\frac{p}{p-1}} \,,
$$
where $\hat H_0$ is the dual norm of $\hat H$. Notice that $\nabla v (x)= 0$ if and only if $x=x_0$.

Now we observe that $Z$ can not have interior points. Since $v$ is not constant then $Z \neq \overline \Omega$. Assume by contradiction that $Z$ has interior points, and consider a connected component $E$ of the interior of $Z$. Let $z \in \partial E \setminus \partial \Omega$ and notice that $z \in \partial \{\nabla v \neq 0\}$. Since $v$ is constant in $E$ and $v$ is of class $C^{1,\alpha}$ then we must have that $v$ is centered at $z$, i.e.
$$
v(x)=c_1^i+c_2^i \hat H_0(x-z)^{\frac{p}{p-1}} 
$$
in any connected component $A_i$ of $\{\nabla v \neq 0\}$ such that $z \in \partial A$. Since $\nabla v = 0$ only at $z$, this implies that any connected component of $\{\nabla v \neq 0\}$ touches $E$ only at $z$. Hence there exist two disjoint connected components $A_1$ and $A_2$ of $\{\nabla v \neq 0\}$ cointaining $z$ on their boundary and such that $(\partial A_1 \cap \partial A_2) \setminus \{z\} \neq \emptyset$. This leads to a contradiction since we must have $c_1^1=c_1^2$ and $c_2^1=c_2^2$, which implies that $A_1$ and $A_2$ are not disjoint.

Since $Z$ has no interior points and $\lambda$ is continuous then we get that $\lambda \neq 0$ is constant, which implies 
$$
\nabla[\hat a(\nabla v(x))]=W(x)=\lambda \, {\rm Id} \qquad \text{in }\Omega\,,
$$ 
and hence 
$$
v(x)=c_1+c_2 \hat H_0(x-x_0)^{\frac{p}{p-1}} 
$$
in $\overline \Omega$. Since $\hat H_0(x) = H_0(-x)$, we find \eqref{Talenti_p} and the proof of Proposition \ref{prop_int_ineq}.
\end{proof}

\section{Proofs of the main results} \label{section_thm}
Once we have Proposition \ref{prop_int_ineq} the proof of Theorems \ref{thm_main}, \ref{thm_main1} easily follows.  

\medskip

\begin{proof}[Proof of Theorems \ref{thm_main} and \ref{thm_main1}] 
As done in the proof of Proposition \ref{prop_mingione}, we approximate $\Omega$ by a sequence of  bounded open convex domains $\{\Omega_k\}$ such that $\Omega_k \subseteq \Omega$ and $\partial \Omega_k$ is smooth. 
Also, we fix a point $\bar x \in \cap_k\Omega_k$, and for $k$ fixed we consider the following problem
\begin{equation}%\label{approx_1bis}
\begin{cases}
\diver(a(\nabla u_k)) + f(u) = 0 & \text{ in } \Omega_k
\\
u_k(\bar x)=u(\bar x)\\
a(\nabla  u_k)\cdot\nu+h(u)=0 & \text{ on } \partial\Omega_k\,.
\end{cases}
\end{equation}
We notice that, for $k$ large enough, the functions $ u_{k}$ are uniformly $C^{1,\theta}$ in every compact subset of $\Omega$, and uniformly H\"older continuous up to the boundary. By considering \eqref{integral_ineq} for $u_k$, using the convexity of $\Omega_k$ and passing to the limit as $k\to \infty$, we obtain that  $u$ satisfies
\begin{multline} \label{integral_ineq_final_u}
(n-1) \int_\Omega u^{\frac{n}{n-p}} H^p(\nabla u) \Phi'(u)\, dx \geq   \int_{\partial \Omega} u^{\frac{n}{n-p} - p^* +1} f(u) h(u)\, d\sigma \\ 
+ n  \int_{\partial \Omega} u^{\frac{n}{n-p} - p^* +1} H^p(\nabla u) \left(  h'(u) - \frac{(p-1)(n-1)}{n-p}\frac{h(u)}{u} \right)\, d\sigma \,.
\end{multline}

Let assume that $\Phi'\leq 0$. We use $h=0$ for Theorem \ref{thm_main} and the condition \eqref{h_condition} for Theorem \ref{thm_main1} in Proposition \ref{prop_int_ineq} and we immediately obtain that 
\begin{equation*}
0\geq   n \int_{\partial \Omega} u^{\frac{n}{n-p} - p^* +1} \mathrm{II}(a^T(\nabla u),a^T(\nabla u))\, d\sigma \geq 0 \,,
\end{equation*}
where the last inequality follows from the convexity of $\Omega$. Hence, the equality case holds in 
\eqref{integral_ineq} which implies that either $u$ is constant or there exists $c_1,c_2 > 0$ and $x_0 \in \overline{\Omega}$ such that
\begin{equation*} %\label{Talenti_p}
u(x) = \left(c_1 + c_2 H_0(x_0-x)^{\frac{p}{p-1}} \right)^{-\frac{n-p}{p}}
\end{equation*}
for any $x \in \overline \Omega$. In the latter case we readily find a contradiction since $a(\nabla u) \cdot \nu + h(u)$  can not vanish on the whole $\partial \Omega$ (with $h=0$ in case of Theorem \ref{thm_main}) and we conclude.
\end{proof}

\noindent It is clear that Corollary \ref{corollary_critical} easily follows from Theorems \ref{thm_main} and \ref{thm_main1} and Lemma \ref{Lemma_Cianchi_prel}. 

\medskip

\noindent Theorem \ref{thm_schiffer} follows again by using Proposition \ref{prop_int_ineq}. In this case, thanks to the fact that $|\nabla u|=0$ on $\partial \Omega$, we do not need to assume that $\Omega$ is convex, since the last term in \eqref{integral_ineq} vanishes thanks to the boundary condition $|\nabla u|=0$.

\end{document}